\documentclass[10pt]{amsart}%
\usepackage{amsmath}
\usepackage{amsfonts}
\usepackage{amssymb}
\usepackage{graphicx}
\usepackage{color}%
\providecommand{\U}[1]{\protect\rule{.1in}{.1in}}
\providecommand{\U}[1]{\protect\rule{.1in}{.1in}}

\theoremstyle{plain}
\numberwithin{equation}{section}
\newtheorem{theorem}{Theorem}[section]
\newtheorem{lemma}[theorem]{Lemma}
\newtheorem{proposition}[theorem]{Proposition}
\newtheorem{corollary}[theorem]{Corollary}

\theoremstyle{definition}

\theoremstyle{remark}
\newtheorem*{claim*}{Claim}
\newtheorem*{example*}{Example}
\newtheorem*{remark*}{Remark}
\newtheorem{remark}{Remark}[section]

\begin{document}
\title[log-concave]{Regularity for a log-concave to log-concave mass transfer problem with near
Euclidean cost}
\author{Micah Warren}
\address{Department of Mathematics, Princeton University\\
Princeton NJ 08540}
\email{mww@math.princeton.edu}
\thanks{The author is partially supported by an NSF grant 0901644.}
\date{\today}
\maketitle

\begin{abstract}
If the cost function is not too far from the Euclidean cost, then the optimal
map transporting Gaussians restricted to a ball will be regular. \ Similarly,
given any cost function which is smooth in a neighborhood of two points on a
manifold, there are small neighborhoods near each such that a Gaussian
restricted to one is transported smoothly to a Gaussian on the other.

\end{abstract}

\section{\bigskip\bigskip Introduction}

This note deals with the regularity of the optimal transportation map, when
the distributions under consideration are close to restricted Gaussians. From the work
of Ma, Trudinger and Wang, ([MTW], [TW]) regularity holds for arbitrary smooth
distributions on nice domains when the cost satisfies the MTW\ A3s condition.
\ \ \ It is established by Loeper [L] that without this MTW condition on the
cost function, one cannot expect regularity for arbitrary smooth
distributions, and the question of regularity is wide open. \ Here we show
that we can find smooth optimal transportation, at least for some very nice
distributions. \ 

We show two results. \ The first is that when the transportation problem
involves distributions somewhat like the standard Gaussian restricted to the
unit ball, then if the cost function is close enough to the Euclidean distance
squared cost, the map must be regular. \ As a corollary, given two points and
any cost which is smooth near these points, we can find very focused
Gaussians, restricted to very small balls near the points, so that the optimal
transport is regular. \ \ \ 

Our method yields a way to compute precisely how close the cost function need
be to Euclidean, or relatedly, how small the balls must be around the given
points. \ Recently other perturbatitive results for regularity of optimal
transport have appeared:\ \ Delano\"{e} and Ge [DG] show regularity
for certain densities on metrics near constant curvature. \ Caffarelli,
Gonzalez and Nguyen\ [CGN] present estimates, when the cost is Euclidean
distance raised to powers other than $2.$

Specifically, let $f,\bar{f}$ \ be functions on regions $\Omega,\bar{\Omega
}\subset%
%TCIMACRO{\U{211d} }%
%BeginExpansion
\mathbb{R}
%EndExpansion
^{n},$ satisfying on $\Omega$
\begin{align}
|Df|  &  \leq1\tag{a1}\\
1  &  \leq\delta\leq D^{2}f\leq2\tag{a2}\\
|D^{3}f|  &  \leq1 \tag{a3}%
\end{align}
and similarly for $\bar{f}$ on $\bar{\Omega}.$

We define the following mass distributions
\begin{equation}
m=e^{-f(x)}\chi_{\Omega} \label{mass}%
\end{equation}

\begin{equation}
\bar{m}=e^{-\bar{f}(\bar{x})}\chi_{\bar{\Omega}} \label{massbar}%
\end{equation}
where we may add a constant to $f$ so that both distributions have the same
total mass.\ 

The region $\Omega$ will be required to have a defining function $h$ so that
on $\Omega=\left\{  h\leq0\right\}  ,$ $h$ satisfies the same three conditions
(a1-3) as $f,$ as well as, along the boundary $\partial\Omega$
\begin{equation}
|Dh|\text{ }\geq1/2, \label{beta}%
\end{equation}
which implies the second fundamental form of the set $\partial\Omega=\left\{
h=0\right\}  $ is bounded by $4.$ \ Similarly define an $\bar{h},~\bar{\Omega
}.$ \ 

A solution of the optimal transportation equation for these densities and a
given cost function $c(x,\bar{x})$ is a function $u(x)$ which satisfies%
\begin{align}
\det w_{ij}  &  =e^{-f(x)}e^{\bar{f}\left(  T(x,Du)\right)  }|\det
c_{is}(x,T(x,Du))|\label{OT}\\
T(x,Du)\left(  \Omega\right)   &  =\bar{\Omega}%
\end{align}
where%
\begin{equation}
w_{ij}=u_{ij}(x)-c_{ij}(x,T(x,Du))=c_{si}T_{j}^{s} \label{this is w}%
\end{equation}
and $T(x,Du)=(T^{1},T^{2},\ldots,T^{n})\subset\bar{\Omega}$ \ is determined by%
\[
u_{i}(x)=c_{i}(x,T(x,Du)).
\]
(Such a solution must also be $c$-convex. \ In our setting, the two notions of
convexity are very close, so we won't belabour this point here, \ see Lemma
\ref{c convexity lemma}.) We will use the following convention: The
derivatives of the cost function in the first variable $x$ will be $i,j,k$
etc. \ The second variable $\bar{x}$ will be denoted by indices $p,s,t,$ etc.
\ Also upper index denotes inverse i.e $c^{is}=(c_{is})^{-1}$.

The cost $c(x,\bar{x})$ will satisfy the standard conditions (A1) and (A2) but
not (A3) (see for example [MTW] section 2.) We will require further that the
second derivatives of the cost satisfy the following assumptions
\begin{equation}
\left\Vert \left(  c^{is}-I\right)  \right\Vert \leq\epsilon_{0}\leq1/20
\tag{c-a1}%
\end{equation}

\begin{equation}
C(n)\left(  \left\Vert D^{3}c\right\Vert +\left\Vert D^{4}c\right\Vert
\right)  \leq\epsilon_{0}\leq1/20 \tag{c-a2}%
\end{equation}
where $C(n)$ is a dimensional constant, and the derivative norms are with
respect to both barred and unbarred directions. \ \ Finally we will require
that the densities are somewhat close to uniform
\begin{equation}
e^{-f(x)}e^{\bar{f}\left(  \bar{x}\right)  }|\det c_{is}|\in\lbrack
\Lambda^{-1},\Lambda] \tag{cm-a3}%
\end{equation}
for all $x,\bar{x}$ $\in\Omega\times\bar{\Omega}$ with
\begin{equation}
\Lambda\leq\left(  \frac{n}{3/2}\right)  ^{n}. \tag{cm-a3b}%
\end{equation}
We are now ready to state our result.

\begin{theorem}
Let $m,\bar{m}$ be the mass densities defined by (\ref{mass}) (\ref{massbar})
with $f,\bar{f}$ satisfying assumptions (a1-3) on regions \ $\Omega,\bar
{\Omega}$ \ whose defining functions also satisfy (a1-3).\ There exists an
$\epsilon_{0}(n)$ such that if the cost function satisfies standard
assumptions (A1) and (A2) and (c-a1,a2) and (cm-a3) hold, then the optimal map
transporting $m$ to $\bar{m}$ is regular. \ \ 
\end{theorem}

\begin{remark}
These conditions are nonvacuous. \ For example take $f,h,\bar{f},\bar{h}$ all
to be
\[
\frac{2}{3}|x|^{2}-\frac{1}{4},
\]
and
\[
c(x,\bar{x})=-x\cdot\bar{x}.
\]
One can check that all the assumptions are satisfied with plenty of room to
perturb any of the problems components. 
\end{remark}

The following theorem will follow by a change of coordinates and rescaling. \ 

\begin{theorem}
\ Let $x_{0},\bar{x}_{0}$ be two points in manifolds $X,\bar{X}$ such that
near $(x_{0},\bar{x}_{0})$ the cost function is smooth and satisfies standard
nondegeneracy conditions (A1)(A2). Then there exists a $\lambda$ large
depending on the cost function, so that the optimal map from the Gaussian
(after a choice of coordinates)\
\[
e^{-\lambda^{2}|x-x_{0}|^{2}/2}\chi_{B_{1/\lambda}(x_{0})}%
\]
to
\[
e^{-\lambda^{2}|\bar{x}-\bar{x}_{0}|^{2}/2}\chi_{B_{1/\lambda}(\bar{x}_{0})}%
\]
is smooth. \ 
\end{theorem}

\begin{remark}
We do not attempt to obtain any sharp results, rather the convenient smallness
assumptions are to minimize crunchiness of the proof.\ Inspection of the proof
will show that our choice of assumptions are robust. There is a rather large gap between what is covered here and the counterexamples, and we have no reason to suspect that these results are near sharp.
\end{remark}

\begin{remark}
We would like to obtain a similar result for complete Gaussians, as Caffarelli
obtained in the Euclidean case in \cite{C2}. \ In fact, it was an attempt to
generalize the calculation in \cite{C2} that led to this result. \ A
limitation of our current method is that we cannot force (cm-a3) to hold on
large regions. \ \ 
\end{remark}

\subsection{Proof Heuristic}

We will solve the problem by continuity, starting with Euclidean cost,
obtaining second derivative estimates using the approach of Urbas \cite{U} and
Trudinger and Wang \cite{TW}, making use of the Ma, Trudinger and Wang \cite{MTW}
calculation together with the calculation of Caffarelli \cite{C2}. \ Making
these methods work in the absense of the MTW condition, we use the following
observation:\ The bound $M$ on the second derivatives will satisfy the
following type of inequality
\begin{equation}
\delta M^{2}-tM^{n+1}-1\leq0. \label{heu}%
\end{equation}
When $t$ is zero, this bounds $M,$ so $M$ is initially bounded. If $t$ is
small it follows that $M(t)$ lies either on a relatively small compact
interval containing $[-1/2,1/2]$ or on a noncompact interval. \ The bound $M(t)$
is changing continuously with $t,$ thus the interval it lies in must not
change, thus from the initial bound we may conclude that for all $t$ in some
interval of fixed size, $M(t)$ is bounded.

The quadratic coefficient $\delta$ in (\ref{heu}) (same $\delta$ as in (a2)) arises
when the target distribution is log-concave, as is the case with Gaussians.
\ This fact is essential to the proof.

\section{Calculations}

Recall the symmetric tensor $w$ (\ref{this is w}).\ \ We use the quantities
defined as follows%
\begin{align*}
W(x)  &  =\sum w_{ii}\sim\max_{i}w_{ij}\sim\left\vert T_{j}^{s}\right\vert \\
\bar{W}(x)  &  =\sum w^{ii}\sim1/\min_{i}w_{ii}\\
C_{3}  &  \geq\left\Vert D^{3}c\right\Vert C(n)\\
C_{4}  &  \geq\left\Vert D^{4}c\right\Vert C(n)\\
\frac{1}{C_{2}}\left\vert \xi\right\vert ^{2}  &  \leq-c^{si}\xi_{i}\xi
_{s}\leq C_{2}\left\vert \xi\right\vert ^{2}%
\end{align*}

From (cm-a3) \ and Newton-McLaurin inequalities, it follows that
\begin{align}
\bar{W},W &  \geq n\frac{1}{\Lambda^{1/n}}\label{2.1}\\
\bar{W} &  \leq\frac{1}{n^{n-2}}\Lambda W^{n-1}\label{2.2}\\
W &  \leq\frac{1}{n^{n-2}}\Lambda\bar{W}^{n-1}\label{2.3}%
\end{align}
and pluggin in (cm-a3b)
\begin{equation}
W,\bar{W}\geq3/2.
\end{equation}
Notice that (a1)(a2)(ca-1)(ca-2) imply the following inequality for any vector
in $%
%TCIMACRO{\U{211d} }%
%BeginExpansion
\mathbb{R}
%EndExpansion
^{n}$
\begin{equation}
\left(  \bar{h}_{st}-c^{kp}c_{kst}\bar{h}_{p}\right)  \xi_{s}\xi_{t}\geq
\frac{9}{10}|\xi|^{2}.\label{cconvex}%
\end{equation}

Throughout this section we will be assuming we have a smooth solution $u$ to
the equation (\ref{OT}) on $\Omega.$ \ Our goal is to prove second derivative estimates.

We make use of the linearized operator at a solution $u,$ from \cite{TW}
defined by%
\[
Lv=w^{ij}v_{ij}-\left(  w^{ij}c_{ij,s}c^{sk}+\bar{f}_{s}(T(x,Du))c^{sk}%
+c^{is}c_{si,p}c^{pk}\right)  v_{k}.
\]

The following has an immediate consequence when maximums occur on the
interior, and is also crucial in the boundary estimates in Section 4. \ \ The
proof is a moderately long calculation and follows by the arguments in
\cite{MTW} .

\begin{lemma}
\label{Lw} \ Suppose $u\left(  x\right)  $ is a solution to (\ref{OT}). Then
\[
Lw_{11}=
\]%
\begin{align*}
&  =w^{ij}\left[  2c_{ijs1}T_{1}^{s}+c_{ijst}T_{1}^{s}T_{1}^{t}-2c_{11is}%
T_{j}^{s}-c_{11st}T_{i}^{s}T_{j}^{t}\right] \\
&  -c_{11p}c^{kp}[-f_{k}+\bar{f}_{s}T_{k}^{s}+c^{is}c_{ist}T_{k}^{t}%
-c_{ijs}w^{ij}T_{k}^{s}-c_{skj}c^{sj}-c^{ti}c^{sj}c_{kst}w_{ij}]\\
&  +\bar{f}_{st}T_{1}^{s}T_{1}^{t}-f_{11}+c^{is}(c_{is11}+2c_{ist1}T_{1}%
^{t}+c_{istp}T_{1}^{t}T_{1}^{p})\\
&  +(c_{1}^{is}+c_{t}^{is}T_{1}^{t})(c_{is1}+c_{isp}T_{1}^{p})\\
&  +\left(  w^{ij}c_{ijp}+\bar{f}_{p}+c^{ij}c_{isp}\right)  c^{pk}\left(
c_{11s}T_{k}^{s}-c_{k1,s}T_{1}^{s}-c_{ks,1}T_{1}^{s}-c_{kst}T_{1}^{s}T_{1}%
^{t}\right) \\
&  -w_{1}^{ij}w_{ij1}.
\end{align*}

\end{lemma}

Applying the maximum principle,

\begin{corollary}
If the largest eigenvalue $W$ of $w$ is attained on the interior, it must
satisfy
\begin{equation}
\frac{\bar{\delta}}{C_{2}}W^{2}-\left(  C_{4}+C_{3}+C_{3}|Df|\right)
W^{n+1}-|D^{2}f|-C(C_{3},C_{4})\leq0. \label{split1}%
\end{equation}

\end{corollary}

\bigskip

The next computation is implicit throughout \cite{TW} sections 2,3 and 4.
\ \ We state it for concreteness. \ 

\begin{lemma}
\label{Lv} Let $v(x)=F(x,T(x,Du)).$ \ Then
\begin{align}
Lv &  =w^{ij}F_{ij}+2F_{is}c^{is}+F_{st}c^{is}c^{jt}w_{ij}\label{eq:Lv}\\
&  +F_{p}\left(  -c^{pk}f_{k}-c^{pk}c_{ks,j}c^{js}-c_{kst}c^{pk}c^{is}%
c^{jt}w_{ij}\right)  \nonumber\\
&  -F_{k}\left(  w^{ij}c_{ijs}c^{sk}+\bar{f}_{s}c^{sk}+c^{is}c_{sit}%
c^{tk}\right) . \nonumber
\end{align}

\end{lemma}

\begin{corollary}
\label{Lh} Given conditions (c-a1) (c-a2) and (a1) (a2)\ on the functions
$f,\bar{f},$ $h,$ and $\bar{h},$ we have
\[
Lh\geq\frac{9}{10}\delta\bar{W}-\frac{11}{10}%
\]%
\[
L\bar{h}(T(x,Du))\geq\frac{9}{10}\delta W-\frac{11}{10}.
\]

\end{corollary}

\bigskip

\subsection{Obliqueness}

We follow the argument from [TW] section 2. \ Defining
\begin{align*}
\gamma &  =Dh\\
\beta &  =\bar{h}_{s}c^{si}\partial_{i}%
\end{align*}
we let%
\[
\chi=h_{k}\bar{h}_{s}c^{sk}=\gamma\cdot\beta.
\]
From Lemma \ref{Lv} with our assumptions we have%
\[
L\chi\leq\bar{W}(\left\vert D^{3}h\right\vert +C_{3}+C_{4})+W(\left\vert
D^{3}\bar{h}\right\vert +C_{3}+C_{4})+C_{5}(n).
\]
Then Corollary \ref{Lh} gives
\[
L\left\{  \chi-\lambda h-\lambda\bar{h}\circ T(x))\right\}  \leq\bar{W}%
(\frac{11}{10}-\frac{9}{10}\lambda)+W(\frac{11}{10}-\frac{9}{10}%
\lambda)+2\frac{11}{10}\lambda+C_{5}(n)
\]
which is negative for $\lambda$ reasonably chosen. (Throughout we are using
bounds (\ref{2.1}) etc, and our initial assumptions.)\ \ \ This function will
then have a minimum at the boundary, precisely\ at the point where $\chi$
achieves a minimum on the boundary, and at this point we have%

\[
\left\{  D\chi-\lambda D(\bar{h}\circ T)-\lambda Dh\right\}  \cdot\frac
{\gamma}{|\gamma|}\leq0
\]
or
\begin{equation}
D\left\{  \chi-\lambda\bar{h}\circ T\right\}  =\tau\gamma\label{tau}%
\end{equation}
for some $\tau\leq\lambda.$

Now computing (following \cite[2.31-2.33]{TW}), using (\ref{cconvex}) \ and
(\ref{this is w}) with our other assumptions including (\ref{beta}) we conclude

\bigskip%
\begin{align}
D\chi\cdot\beta &  =c^{ti}\bar{h}_{t}\left(  h_{ki}c^{sk}\bar{h}_{s}%
+h_{k}\left(  c_{i}^{sk}+c_{p}^{sk}T_{i}^{p}\right)  \bar{h}_{s}+h_{k}%
c^{sk}\bar{h}_{sp}T_{i}^{p}\right) \nonumber\\
&  =h_{ki}\beta^{k}\beta^{i}+c^{ti}\bar{h}_{t}h_{k}\bar{h}_{s}c_{i}%
^{sk}+c^{ti}\bar{h}_{t}h_{k}T_{i}^{p}(\bar{h}_{sp}c^{sk}-\bar{h}_{s}%
c^{sm}c^{rk}c_{mrp})\nonumber\\
&  =h_{ki}\beta^{k}\beta^{i}+c^{ti}\bar{h}_{t}h_{k}\bar{h}_{s}c_{i}^{sk}%
+h_{k}\bar{h}_{t}T_{a}^{t}c^{pa}c^{rk}(\bar{h}_{rp}-\bar{h}_{s}c^{sm}%
c_{mrp})\label{adf}\\
&  \geq|\beta|^{2}\delta-C_{3}\geq\frac{1}{5}\delta,\nonumber
\end{align}
The third term in (\ref{adf}) can be expressed as an inner product $g$ of the
gradients of the functions $h(x)$ and $\bar{h}\circ T(x),$ which are both
multiples of the outward normal,\ where
\[
g(\xi,\nu)=(\bar{h}_{rp}-\bar{h}_{s}c^{sm}c_{mrp})c^{rk}c^{pa}\xi_{k}\nu_{a}.
\]
Thus
\begin{align*}
\tau\gamma\cdot\beta &  =D\chi\cdot\beta-\lambda D(\bar{h}\circ T)\cdot\beta\\
&  \geq\delta/5-\lambda\bar{h}_{s}T_{i}^{s}c^{it}\bar{h}_{t}\\
&  =\delta/5-\lambda w_{\beta\beta}.
\end{align*}
Thus from $\tau \leq \lambda,$
\begin{equation}
\lambda\chi\geq\delta/5-\lambda w_{\beta\beta}. \label{chi1}%
\end{equation}
Using symmetry (replacing all quantities with barred quantities we find the
problem does not change, again see \cite{TW} and Lemma \ref{c convexity lemma}%
)\ , we may assume
\begin{equation}
\lambda\chi\geq\delta/5-\lambda\bar{w}_{\gamma\gamma}. \label{chi2}%
\end{equation}
Then, using the Urbas formula \cite{U},\ \cite[2.13]{TW}
\[
(\beta\cdot\gamma)^{2}=w^{ij}\gamma_{i}\gamma_{j}w_{\beta\beta}%
\]
or
\begin{equation}
\chi^{2}=\bar{w}_{\gamma\gamma}w_{\beta\beta} \label{chi3}%
\end{equation}
we have combining (\ref{chi1}) (\ref{chi2}) and (\ref{chi3})%
\begin{equation}
\chi\geq\frac{\delta}{10\lambda}=\theta. \label{theta}%
\end{equation}

\begin{corollary}
\label{delta} The following holds, regarding the angle between $\beta$ and
$\gamma$ $\ $%
\[
\angle(\beta,\gamma)\leq\Delta<\pi/2.
\]

\end{corollary}

\subsection{cost-convexity}

\begin{lemma}
\label{c convexity lemma}Suppose $u\left(  x\right)  $ is a solution to
(\ref{OT}) on a domain in $%
%TCIMACRO{\U{211d} }%
%BeginExpansion
\mathbb{R}
%EndExpansion
^{n}$. \ \ If $D^{2}u\geq2\epsilon_{0},$ and the cost function differs from
the Euclidean cost function by less than then $\epsilon_{0}$ in $C^{2},$ then
$u$ is $c$-convex, and the mapping $T(x,u)\ $is one to one.
\end{lemma}

\begin{proof}
\ Suffice to consider the $c=-x\cdot y+\phi(x,y),$ where $\phi$ is small in
$C^{2}(\Omega).$ At a point $x_{0}$, we have $Du(x_{0})=Dc\left(
x_{0},T(x_{0},Du\right)  )=\ -T(x_{0},Du)+D\phi(x_{0},T(x_{0})).$ \ At another
point, $x_{1}$ \
\[
\langle Du(x_{1})-Du(x_{0}),x_{1}-x_{0}\rangle\geq2\epsilon_{0}|x_{1}%
-x_{0}|^{2}.
\]
Now suppose that $u$ is not strictly $c$-convex. \ \ Clearly the issue would
have to be nonlocal, as locally,%
\[
D^{2}u-D^{2}c\geq2\epsilon_{0}-\epsilon_{0}>0.
\]
Thus we can assume that there is a point $x_{0}$ and a locally supporting cost
function
\[
c_{y_{0}}(x)=-x\cdot T(x_{0})+\phi(x,T(x_{0}))
\]
which contacts $u$ from below near $x_{0}$ but touches $u$ (possibly
transversely) at a point $x_{1}.$ \ It follows that
\[
\langle Dc_{y_{0}}(x_{1})-Dc_{y_{0}}(x_{0}),x_{1}-x_{0}\rangle\geq\langle
Du(x_{1})-Du(x_{0}),x_{1}-x_{0}\rangle
\]
that is%
\[
\left\Vert D^{2}\phi\right\Vert _{C^{1,1}}|x_{1}-x_{0}|^{2}\geq2\epsilon
_{0}|x_{1}-x_{0}|^{2}%
\]
\ a contradiction. \ It follows that $u$ is $c$-convex \ and $T$ is one to one. \ 
\end{proof}

\bigskip

\subsection{Boundary Estimate}

\bigskip Let
\[
M=\max_{\left\vert e\right\vert =1,e\in T_{x}\Omega}w_{ee}%
\]
be the maximum of all eigenvalues $W$ over all of $\Omega$. Throughout this
section we will assume that the maximum occurs on the boundary.

Recalling (\ref{2.3}) and Lemma \ref{Lv}, we may choose a $C_{6}$ so that
\[
L(C_{6}M^{n-2/n-1}h-\bar{h}(T(x,Du))\geq0.
\]
Since $h,\bar{h}$ both vanish on the boundary, the derivatives must satisfy
\[
D_{\beta}\bar{h}\circ T(x,Du)\leq C_{6}M^{n-2/n-1}%
\]
that is%
\[
\bar{h}_{s}T_{i}^{s}\beta_{i}=\bar{h}_{s}c^{sj}w_{ij}\bar{h}_{t}%
c^{ti}=w_{\beta\beta}\leq C_{6}M^{n-2/n-1}.
\]

\begin{lemma}
\label{brendle} At a point $x_{0}$ on the boundary $\partial\Omega,$%
\bigskip\ suppose $w_{ee}\leq M$ for unit directions $e$ which are tangential
to the boundary. If $z$ is any vector in $T_{x_{0}}\Omega,$ then
\[
w_{zz}\leq M|\hat{z}|^{2}+\frac{1}{\theta^{2}}\langle z,\nabla h\rangle
^{2}w_{\beta\beta}.
\]
where
\[
\hat{z}=z-\frac{\gamma\cdot z}{\gamma\cdot\beta}\beta=z-y,
\]
and $\theta$ is defined by (\ref{theta}).
\end{lemma}

\begin{proof}
Dotting with $\gamma$ verifies $\hat{z}$ is tangential, thus
\[
0=\partial_{\hat{z}}\bar{h}\circ T(x,Du)=\bar{h}_{s}T_{j}^{s}\hat{z}_{j}%
=\bar{h}_{s}c^{is}w_{ij}\hat{z}_{j}.
\]
Now
\[
w_{zz}=w_{\hat{z}\hat{z}}+2w_{\hat{z}y}+w_{yy}%
\]
but
\[
w_{\hat{z}y}=w_{ij}\hat{z}_{j}\bar{h}_{s}c^{is}=0
\]
so
\[
w_{zz}\leq M|\hat{z}|^{2}+\left(  \frac{\gamma\cdot z}{\gamma\cdot\beta
}\right)  ^{2}w_{\beta\beta}.
\]

\end{proof}

Now suppose that the maximum tangential derivative $w_{11}=M^{T}$ happens at a
point $x_{0}$, where $e_{1}$ is a tangential direction. \ Define the function
\[
\eta=w_{11}-M^{T}|\hat{e}_{1}(x)|^{2}-\ C_{6}\frac{1}{\theta^{2}}\langle
e_{1},\nabla h(x)\rangle^{2}M^{n-2/n-1}+C_{7}(M+1)(h+\bar{h}\circ T)
\]
where\
\[
|\hat{e}_{1}(x)|^{2}=\left\vert e_{1}-\frac{h_{1}(x)}{\xi(\bar{h}_{s}%
(T)c^{sk}h_{k}(x,T))}\beta\right\vert ^{2}%
\]
with $\xi$ a smooth function satisfying $\xi(t)=t$ for $t>\theta/2,$ and
$\xi(t)\geq\theta/4.$ \ $\ $Now computing, using Lemma \ref{Lw} and
(\ref{2.2})%
\begin{align*}
L\eta &  \geq\delta w_{11}^{2}-\left(  C_{4}+C_{3}\right)  \bar{W}%
W^{2}-C(n)-M\left\vert L|\hat{e}_{1}(x)|^{2}\right\vert )-\ C_{6}\frac
{1}{\theta^{2}}|L\ \langle e_{1},\nabla h(x)\rangle^{2}|M^{n-2/n-1}\\
&  +C_{7}(M+1)\left\{  \frac{9}{10}\delta\left(  \bar{W}+W\right)
-2(1+\mu)\right\}
\end{align*}
and using (considering Lemma \ref{Lv})
\begin{align*}
\left\vert L|\hat{e}_{1}(x)|^{2}\right\vert  &  \leq C_{8}(\bar{W}+1+W)\\
\left\vert LC_{6}\langle e_{1},\nabla h(x)\rangle^{2}\right\vert  &  \leq
C_{8}(\bar{W}+1+W).
\end{align*}
we may choose
\[
C_{7}=C_{8}+\left(  C_{4}+C_{3}\right)  \left(  \bar{M}+M\right)
\]
so that
\[
L\eta\geq0.
\]

Next we show a lower bound on $D_{\beta}w_{11}(x_{0}).$ \ First, observe that
due to Lemma \ref{brendle}, $\eta$ has a maximum at $x_{0}.$ \ It follows from
the Hopf maximum principle that $D\eta\cdot\beta=\nu\gamma\cdot\beta$
\ $\geq0.$ \ Thus (recalling $h_{1}(x_{0})=0$)%
\begin{align}
D_{\beta}w_{11}(x_{0}) &  \geq M^{T}D_{\beta}|\hat{e}_{1}|^{2}+D_{\beta}%
C_{6}\langle e_{1},\nabla h(x)\rangle^{2}M^{n-2/n-1}\label{eq:42}\\
&  -\left\{  C_{8}+\left(  C_{4}+C_{3}\right)  \left(  \bar{M}+M\right)
\right\}  M(D_{\beta}h+D_{\beta}H)\nonumber\\
&  \geq-C(n)M^{T}-\left\{  C_{8}+\left(  C_{4}+C_{3}\right)  \left(  \bar
{M}+M\right)  \right\}  C_{6}(n)(1+M^{2n-3/n-1}).\nonumber
\end{align}

\bigskip Finally we will derive a relation between the maximum $M$ of all
eigenvalues of $w$ and for tangential eigenvalues $M^{T}$. Go to the point
where the maximum of all eigenvalues for $w$ happens. (Again, in this section
we\ assume this happens along the boundary.) \ We diagonalize
$w=diag(M,\lambda_{2},\ldots\lambda_{n})$ with respect to some coordinates
$e_{1,}\ldots e_{n},$ choosing $e_{1}\cdot\gamma\geq0.$\ Now
\[
w_{\beta\beta}=(\beta\cdot e_{1})^{2}M+(\beta\cdot e_{2})^{2}\lambda
_{2}+\ldots(\beta\cdot e_{n})^{2}\lambda_{n}\leq C_{6}(n)M^{n-2/n-1}%
\]
thus%

\begin{equation}
(\beta\cdot e_{1})^{2}\leq C_{6}(n)M^{-1/n-1}.\label{alternative:1}%
\end{equation}
It follows that there is a $C_{10}$ depending on $C_{6}(n)$ and $\Delta,$
(recall Corollary \ref{delta}) such that if $M\geq C_{10},$ then
\[
\left\vert \angle(\beta,e_{1})-\pi/2\right\vert <\frac{1}{2}(\pi/2-\Delta)
\]
in particular
\[
\angle(\gamma,e_{1})\geq\frac{1}{2}(\pi/2-\Delta).
\]
Thus the length of projection of the maximum eigenvector of $w$ onto the
tangent plane is at least some value $\sigma M$ depending on $\Delta.$ So we
may assume that either $M\leq C_{10},$ or the maximum tangential value $M^{T}$
satisfies $M^{T}$ $\geq\sigma M.$

\begin{proposition}
Suppose that the global maximum for $w~$is attained along the boundary. \ Then
if $M\geq$ $C_{10}$, $M$ must satisfy%
\begin{equation}
M^{2}-\left(  C_{4}+C_{3}\right)  M^{n+1}\leq C_{11}%
\end{equation}

\end{proposition}

\begin{proof}
Differentiating $\bar{h}\circ T(x,Du)$ twice tangentially,
\begin{align}
\partial_{11}\bar{h}\circ T(x,Du)) &  =\bar{h}_{s}T_{11}^{s}+\bar{h}_{st}%
T_{1}^{s}T_{1}^{t}=-\langle\nabla\bar{h}\circ T,II(1,1)\rangle\label{diff 11}%
\\
&  =\bar{h}_{p}\left(  c^{pk}w_{11,k}+c^{pk}c_{11,s}T_{k}^{s}-c^{pk}%
c_{k1,s}T_{1}^{s}-c^{pk}c_{ks,1}T_{1}^{s}-c_{kst}c^{pk}T_{1}^{s}T_{1}%
^{t}\right)  \nonumber\\
&  +\bar{h}_{st}c^{si}w_{i1}c^{ti}w_{j1}%
\end{align}
using \cite[4.11]{MTW}. \ Now using $\bar{h}_{p}c^{pk}w_{11,k}=w_{11,\beta}$,
(\ref{eq:42}) and the discussion in the previous paragraph we conclude that if
$M\geq C_{10},$
\begin{align*}
&  \delta\sigma M^{2}-C(n)M_{T}-\left\{  C_{8}+\left(  C_{4}+C_{3}\right)
\left(  \bar{M}+M\right)  \right\}  C_{6}(n)M^{2n-3/n-1}-C_{3}W^{2}\\
&  \leq C_{6}M^{n-2/n-1}.
\end{align*}
Using Young's inequality to clean up the expression, we have
\begin{equation}
M^{2}-\left(  C_{4}+C_{3}\right)  M^{n+1}\leq C_{11}.\label{split:2}%
\end{equation}

\end{proof}

\section{Proof of Theorem}

We now go through the alternatives and make our choice of constants, in order
to bound $w$ and consequently $D^{2}u$. 

First, if the maximum happens in the interior, then (\ref{split1})%
\begin{equation}
M^{2}-\left(  C_{4}+C_{3}\right)  M^{n+1}\leq C_{12}.\label{b1}%
\end{equation}
If not, then either (\ref{split:2})
\begin{equation}
M^{2}-\left(  C_{4}+C_{3}\right)  M^{n+1}\leq C_{11}\label{b2}%
\end{equation}
or
\begin{equation}
M\leq C_{10},\label{b3}%
\end{equation}
by the discussion surrounding (\ref{alternative:1}).

So we simply must choose $\left(  C_{4}+C_{3}\right)  $ small enough, say
\[
\left(  C_{4}+C_{3}\right)  \leq\varepsilon_{0}%
\]
so that the noncompact region defined by (\ref{b1}) does not intersect the
compact regions defined by (\ref{b2}) and (\ref{b3}), similarly for the
noncompact region defined by (\ref{b2}). \ \ Further, in order to have
$c$-convexity, we must assume that the conditions of Lemma
\ref{c convexity lemma} are satisfied. \ The upper bounds in the above
alternatives provide lower bounds on the Hessian, so we choose $C_{3}$ small
enough so that Lemma \ref{c convexity lemma} is satisfied. \ 

Now by the theory of Delanoe [D], Caffarelli [C1] and Urbas [U] we have a
classical solution to the problem for distance squared
\[
c^{0}(x,y)=|x-y|^{2}\text{/2}%
\]
in Euclidean space. \ 

We use the method of continuity. Openness is provided by Theorem 17.6 in GT,
where we set%
\[
G:C^{2,\alpha}(\Omega)\times\lbrack0,1]\rightarrow C^{0,\alpha}\left(
\Omega\right)  \times C^{1,\alpha}\left(  \partial\Omega\right)
\]
with
\begin{align*}
G(u,t) &  =\\
&  \left(
\begin{array}
[c]{c}%
\ln\det\left[  u_{ij}-c_{ij}^{(t)}(x,T^{(t)}(x,Du)\right]  -h(x)+\bar
{h}(T^{(t)}(x,Du))-\ln\det\left[  c_{is}^{(t)}(x,T^{(t)}(x,Du))\right]  ,\\
\bar{h}(T^{\left(  t\right)  }(x,Du))
\end{array}
\right)
\end{align*}
where the cost function is changing from Euclidean to $c$ as
\[
c^{(t)}=(1-t)c^{0}+tc
\]
and $T^{(t)\text{ }}$defined by
\[
Dc^{(t)}(x,T^{(t)}(x,Du))=Du.
\]
Our initial solution $u_{0}$ is smooth , so it satisfies the above estimates
(\ref{b1}, etc) with $C_{3,}C_{4}=0.$ \ These bounds change continuously with
$t$ so $D^{2}u$ must stay in the compact components of (\ref{b1}) (\ref{b2})
and (\ref{b3}). \ As is standard for this problem, we cite [LT] to obtain the
$C^{2,\alpha}$ estimates.\ By [GT] Theorem 17.6, we have openness in $t,$ and
the estimates give us closedness as long as $\left\vert D^{4}c^{\left(
t\right)  }\right\vert ,\left\vert D^{3}c^{\left(  t\right)  }\right\vert
\ \leq$\ $\varepsilon_{0}$. \ \ This completes the proof of Theorem 1.1.

\section{Theorem 2}

\bigskip First we employ a change of coordinates so that
\[
c_{is}(x_{0},\bar{x}_{0})=-I_{n}.
\]

\begin{proof}
Then, on a\ product of very small balls $B_{1/\lambda}(x_{0})\times
B_{1/\lambda}(\bar{x}_{0})$ we have
\[
\frac{1}{C_{2}}\left\vert \xi\right\vert ^{2}\leq-c^{si}\xi_{i}\xi_{s}\leq
C_{2}\left\vert \xi\right\vert ^{2}%
\]
for some $C_{2\text{ }}$ near $1,$ and $\left\vert D^{3}c\right\vert
,\left\vert D^{4}c\right\vert \leq C$ which may be large but finite.

We now rescale and consider the following problem on $B_{1}(0)\times
B_{1}(\bar{0})$: \ Let
\[
c^{(\lambda)}(y,\bar{y})=\lambda^{2}c(\frac{y}{\lambda},\frac{\bar{y}}%
{\lambda})
\]
be the cost function, and let the distributions to be transported be
Gaussians, satisfying (a1-3) on $B_{1}(0),B_{1}(\bar{0}).$

This cost function $c^{(\lambda)}$ \ now satisfies the conditions in our first
theorem, as we see that choosing $\lambda$ large enough will make the third
and fourth derivatives arbitrarily small. \ \ 

It follows by Theorem 1.1 that the solution to this rescaled optimal
transportation problem is smooth. \ \ \ However, the coordinate change and
"change of currency" do not change the underlying optimal transportation
problem. \ Thus we also have smoothness for the solution of the problem
sending
\[
m=e^{-\lambda^{2}|x-x_{0}|^{2}/2}\chi_{B_{1/\lambda}(x_{0})}%
\]
to
\[
\bar{m}=e^{-\lambda^{2}|\bar{x}-\bar{x}_{0}|^{2}/2}\chi_{B_{1/\lambda}(\bar
{x}_{0})}.
\]
This completes the proof.
\end{proof}

\end{document}